\documentclass[leqno,10pt]{amsart}
 \usepackage[square,sort,comma,numbers]{natbib}
\usepackage{palatino,ulem}
\usepackage{amsfonts,amsmath} 
\usepackage{amsthm,soul}
\usepackage{graphicx,enumerate}
\usepackage[hmarginratio=1:1,hscale=0.75]{geometry}
\usepackage[latin1]{inputenc}
\usepackage[all]{xy}

\usepackage{amssymb,mathtools}
\usepackage{stmaryrd}
\usepackage{hyperref}
\usepackage{pdfsync}
\usepackage{color}
\usepackage[T1]{fontenc} 
\usepackage{lmodern}

\newtheorem{counter}{Counter}[section]
\newtheorem{lem}[counter]{Lemma}
\newtheorem{defn}[counter]{Definition}
\newtheorem{thm}[counter]{Theorem}

\newtheorem{cor}[counter]{Corollary}
\newtheorem{problem}[counter]{Problem}
\numberwithin{equation}{section}

\newcommand{\R}{\mathbb{R}}
 
\renewcommand{\H}{\mathcal{H}} 
\renewcommand{\L}{\mathcal{L}}

\newcommand{\C}{\mathcal{C}} 
\newcommand{\A}{\mathcal{A}}
\newcommand{\E}{\mathcal{E}}

\newcommand{\lra}{\longrightarrow} 
\newcommand{\Lra}{\Longrightarrow}
\newcommand{\ra}{\rightarrow} 
\newcommand{\da}{\downarrow}

 \newcommand{\sse}{\subseteq}
\renewcommand{\~}{\tilde}

 \newcommand{\p}{p} 
 \newcommand{\sh}{\sharp}
 
\newcommand{\rhu}{\;{\overset{*}{\rightharpoonup}}\; }
\newcommand{\rdh}{\; {\overset{{{d_\H}}}{\ra}} \;}
\newcommand{\fal}{\forall}
\newcommand{\8}{\infty} 
\newcommand{\te}{\textnormal}
\newcommand{\vph}{\varphi}
\newcommand{\vep}{\varepsilon} 

\newcommand{\al}{\alpha}
 
\newcommand{\gm}{\gamma}
 \newcommand{\ggm}{\Gamma}
\newcommand{\sm}{\Sigma}
 
\newcommand{\la}{\lambda} 
\renewcommand{\vartheta}{\Theta}

\renewcommand{\th}{\theta}
\newcommand{\disp}{\displaystyle}

\DeclareMathOperator{\supp}{{\operatorname{supp}}}
\DeclareMathOperator{\diam}{{\operatorname{diam}}}
\DeclareMathOperator{\argmin}{{\operatorname{argmin}}}
\DeclareMathOperator{\Span}{{\operatorname{Span}}}

\newenvironment{listi}
  {\begin{list}
 {(\roman{broj})}
{ \usecounter{broj}}}
{   \end{list} }
\newcommand{\red}{\color{black}}

\definecolor{mygreen}{rgb}{0.1,0.75,0.2}
\newcommand{\grn}{\color{black}}
\newcommand{\nc}{\normalcolor}
\newcommand{\D}{\diam \supp(\mu)}
\renewcommand{\d}{\,{\operatorname{d}}}
\newcommand{\rca}{\,\overset{\C}{\ra}\,}
\newcommand{\EE}{\E_\mu^{\la,\vep,\p}}

\begin{document}
\newcounter{broj}

\title[Curvature penalization]{Average-distance problem with curvature penalization
	for data parameterization: regularity of minimizers}

\author{Xin Yang Lu} \address{Department of Mathematics and Statistics, Lakehead University,
	Thunder Bay, ON, P7B5E1, \textsc{Canada},  
	AND
	Department of Mathematical Sciences, McGill University,
	Montr\'eal, QC, H3A0B9, \textsc{Canada}, email: xlu8@lakeheadu.ca}

\author{Dejan Slep\v{c}ev} \address{Department of Mathematical Sciences, Carnegie Mellon University,
	Pittsburgh, PA, 15213, \textsc{United States}, email: {slepcev@math.cmu.edu}}

\date{}



\begin{abstract}
We propose a model for finding one-dimensional structure in a given measure. 
Our approach is based on minimizing an objective functional which combines 
 the average-distance functional to measure the quality of the approximation and penalizes the curvature, similarly to the elastica functional.
Introducing the curvature penalization overcomes some of the shortcomings of the average-distance functional, in particular the lack of regularity of minimizers.
We establish existence, uniqueness and regularity of minimizers of the proposed functional. In particular  we establish $C^{1,1}$ estimates on the minimizers.
 \end{abstract}
 
\maketitle

%
%

\section{Introduction}
The average-distance problem was introduced by 
{Buttazzo, Oudet} and {Stepanov} in \cite{BOS}. A number of 
properties of the solutions were established by Buttazzo and Stepanov in \cite{BS1,BS2}, and by
Paolini and Stepanov in \cite{PaoSte}. An alternative variant,
often referred to as {\em ``penalized formulation''}, was introduced by
{Buttazzo, Mainini} and {Stepanov} in \cite{BMS}. \grn Here  we consider  a more general form where 
we consider the general power of the distance $p \in [0, \infty)$, while the works above focused on $p=1$: \nc
\begin{problem}\label{main pen}
Given \grn $p \in [0, \infty)$ \nc, $d\geq 2$, a finite measure $\mu\geq 0$ with compact support in $\R^d$, and $\lambda>0$,
minimize
\begin{equation*} 
G_\mu^\la(\sm):= \int_{\R^d}{d(x,\sm)^p} \d\mu(x)+\lambda \H^1(\sm), 
\qquad {d(x,\sm):= \inf_{y\in\sm} |x-y|},
\end{equation*}
with the unknown $\sm$ varying in the family
$$\A:=\{\sm\sse\R^d: \sm\ \text{compact, path-wise connected and } \H^1(\sm)<+\8 \}.$$
\end{problem}
Problem \ref{main pen} finds applications in urban planning (see
for instance
Buttazzo, Pratelli and Stepanov \cite{BAdd1},
Buttazzo and Santambrogio \cite{BAdd3}; a review is available in
Buttazzo and Santambrogio \cite{BAdd2}, and in the monograph by
Buttazzo, Pratelli, Solimini and Stepanov \cite{BAdd4}) and  is closely related to several functionals in data analysis which look for low dimensional structures in data by minimizing the loss combined with a regularization term  (see for instance  Smola, Mika, Sch\"olkopf and Williamson \cite{Smola}). 
In particular it is closely related to some functionals that arise as regularizations of principal curves, as we discuss below. \nc



To consider the application of the average-distance problem in data parameterization consider a compactly supported measure $\mu$ describing the data distribution. Denote by $\sm$ (the unknown) an one-dimensional object which parameterizes the, potentially noisy, data distribution. Then
 ${\disp{\int_{\R^d}  d(x,\sm)^p \d\mu}}$ represents the approximation error, while $\lambda\H^1(\sm)$ is a cost
associated to the complexity of such representation. Minimizing $G_\mu^\lambda$ 
is thus determining the
``best'' parameterization, balancing approximation error and complexity. 
In practice one would often deal with the empirical measure $\mu_n$ of a sample of $\mu$, rather that $\mu$ itself, but note that in our setup $\mu$ can be a general measure, and thus one can consider $\mu_n$ in the place of $\mu$ when considering the basic propertied of the minimizers when the empirical loss is used. 

It is often of interest to use even simpler objects to approximate measures, namely parameterized curves. 
 Let
$$\C^*:=\{\vph^*:[0,1]\lra \R^d: \vph^* \text{ Lipschitz and }
|(\vph^*)'| \text{ is constant } \L^1\text{-a.e.}\}. $$
For curves $\vph^*\in\C^*$, 
we define its ``length'' $L(\vph^*)$ 
 as its total variation:
 $$L(\vph^*):=\|\vph^*\|_{TV([0,1])}=\sup_n\left(\sup_{0\leq t_0 <\cdots<
 t_{n-1}<t_n=1} \sum_{j=1}^n |\vph^*(t_{j})-\vph^*(t_{j-1})|\right).$$ 
For the sake of simplicity, 
in the following we will work with arc-length (instead
of constant speed) parameterizations. Let
$$\C:=\{\vph:[0,a]\lra \R^d: a\geq 0,\  \vph \text{ Lipschitz
and } |\vph'|=1 \ \L^1\text{-a.e.}\}. $$
Similarly, given $\vph\in \C$, $\vph:[0,a]\lra \R^d$,
we define its ``length''
$L(\vph)$ as its total variation.
The map ``reparameterization by constant speed''
\begin{equation}\label{Phi}
\Phi:\C\lra \C^*,\qquad \Phi(\vph) = \vph^*,\qquad
\vph^*:[0,1]\lra \R^d,\quad \vph^*(t):=\vph(ta)
\end{equation}
will be frequently used.
Thus, by construction, the domain of $\vph$ is $[0,L(\vph)]$.
More details about the space $\C$, including its topology, will be presented in Section \ref{prel}.
For any given curve $\psi$, we denote by $\ggm_\psi$
 its { image} (which is independent of the parameterization).

The average-distance problem \grn  for parameterized curves is the following: \nc
\begin{problem}\label{main1}
 Given \grn $p \in [1, \infty)$ \nc, $d\geq 2$, a finite measure $\mu\geq 0$ with compact support in $\R^d$,  and $\lambda>0$, 
  minimize
$$
E_\mu^\lambda(\vph):=  \int_{\R^d} d(x,\ggm_\vph)^p  \d\mu +\la L(\vph).$$
 over all $\vph\in \C$.
\end{problem}
\grn In mathematical literature on average-distance problem \cite{BMS, BOS, BS1,BS2, PaoSte} the power $p=1$ is considered, while in applications to machine learning $p=2$ is the most common. \nc

Problem \ref{main1} is an alternative to the classic ``principal curve''  {   introduced by Hastie and Stuetzle \cite{Add2} and further studied by  Duchamp and Stuetzle \cite{AddGeom}, for discovery and parameterization of one dimensional structures in data. }
Tibshirani \cite{Add4} introduced principal curves with curvature penalty. 
In K\'egl, Krzyzak, Linder, Zeger \cite{kegl2000}, the authors  studied the principal curves with length constraint which is posed as a minimization of  the mean squared distance and is hence similar to 
Problem \ref{main1}. They also introduced an iterative algorithm to find the optimal cuurves. Subsequently,
Biau and Fisher \cite{biau2011} proposed an equivalent formulation
of principal curves, with the advantage of the self-consistency condition being explicit,
and studied the problem with either length or curvature bounds.
{  More recently, Delattre and Fischer \cite{Dellatre-Fischer}
proved several theoretical properties, such as lack of self consistency,
 bounded curvature as a measure and absence of multiple points,
of minimizers of the principal curves problem with length constraint. In terms of the regularity of curvature these results are similar (while the techniques are different) to results in \cite{LuSlepcev, LuSlepcev16} where the penalized problem was studied.
 Kirov and one of authors \cite{KirSle17} relaxed the Problem \ref{main1} to allow for multiple curves, and developed an efficient algorithm for both Problem \ref{main1} and the relaxed problem.
} \nc
Delicado \cite{delicado} introduced a new notion of principal curves based on the so-called ``principal oriented points''.
Gerber and Whitaker \cite{gerber2013} studied a relaxed notion of principal curves, 
(i.e. ``weak principal curves'') and have shown it to be equal to the conditional expectation curve  of a projection distance functional. 
A manifold learning algorithm based on the gradient descent of the projection
distance functional (see e.g. \cite[(5)]{gerber2009dimensionality}) 
was then formulated in Gerber, Tasdizen, Whitaker \cite{gerber2009dimensionality}.
Another notion of principal curves, defined based only on local quantities such as the
 gradient and Hessian matrix of an estimate of the  probability density
was proposed in Ozertem and Erdogmus \cite{ozertem2011locally}.
 Furthermore, Problem \ref{main1} is
related to the lazy traveling salesman problem, see for instance Polak and Wolanski \cite{WolPol}.
\medskip

\grn Let us recall what is known about the obstacles to regularity  of minimizers of Problem \ref{main1}: \nc
\begin{enumerate}
\item[($*$)] In \cite{Slepcev} it has been proven that, even if the reference measure $\mu$ 
 satisfies
$\mu\ll\L^d$ and $\d\mu/\d\L^d$ is $C^\8$ regular,
  Problem \ref{main pen} may still admit minimizers which are simple curves failing to be $C^1$ regular.
Moreover for any corner (i.e., point where $C^1$ regularity fails), 
 a positive amount of mass is projected on it (see 
 \cite[Lemma~2.1]{LuSlepcev} for a detailed
discussion about ``projections'').

\item[($**$)] In \cite{Lu} it has been proven that in general,
even when the reference measure $\mu$ has the form
$\mu=\sum_{j=1}^\8 a_j \chi_{A_j}$ (here $\chi$ denotes the characteristic function
of the subscripted set),
for suitable choices of parameters $\la$, $\{a_j\}$ and subsets $\{A_j\}$ ($j=1,2,\cdots$),
 Problem \ref{main pen} can admit minimizers
which are simple curves for which $C^1$ regularity fails on a non closed set.
\end{enumerate}
Noting that if a minimizer $\sm$ of Problem \ref{main pen} is 
a simple curve, then it admits a parameterization $\vph\in \C $
minimizing Problem \ref{main1}, it follows that
the formulation of Problem \ref{main1} presents some drawbacks
when used in data parameterization:

\begin{itemize}
\item fact ($*$) implies that a positive fraction of the data
{ could be} projected onto a single point. This is undesirable in data parameterization, 
since it results in a loss of information.

\item In \cite[Lemma~3.1]{LuSlepcev} 
however it has been proven that the aforementioned issue is somewhat inevitable,
 since for any {\em endpoint}, at least mass
$\la>0$ is projected on it. Nevertheless, this implies
that there are at most $\lfloor 1/\la\rfloor$ endpoints. Thus, from a practical point of view,
endpoints can be ``singled out''
and analyzed with ad hoc arguments, mitigating
the aforementioned issue. However fact ($**$) proves that the set of corners 
can be quite irregular, which
  makes ``singling out''
corners much more difficult than ``singling out'' endpoints. 
\end{itemize}
It has been proven, by Lemenant in \cite{Lem1} and by the authors in \cite{LuSlepcev}, \grn as well as by Delattre and Fischer \cite{Dellatre-Fischer} for the constrained problem,  \nc 
that regularity and curvature
of any given subset of a minimizer is related to \grn the amount for mass when $p=1$ and otherwise the $p-1$-moment of mass that projects of the subset considered. \nc
We propose a modification of Problem \ref{main1}, by adding
a term penalizing the integrated squared curvature. 
We will consider the general case, where the power $p$ belongs to $[1, \infty)$.

\begin{problem}\label{main}
 Given $d\geq 2$, a {  finite measure $\mu \geq 0$ compactly supported in $\R^d$}, parameters $\lambda,\vep>0$, 
 $\p\geq 1$, minimize
\begin{equation}
\E_\mu^{\lambda,\vep,\p}(\vph):=
\begin{cases}
\int_{\R^d} d(x,\ggm_\vph)^\p  \d\mu +\la L(\vph) + \vep 
\int_0^{L(\vph)} |\kappa_\vph|^2 \d\L^1, & \text{if }\vph\in H^{2}([0,L(\vph)];\R^{ d}) {\cap \C} ,\\
+\8& \text{otherwise},
\end{cases}
\label{curv}
\end{equation}
 %
with $\kappa_\vph { := \frac{d}{dt}\frac{\vph'}{\|\vph'
	\|}} :[0,L(\vph)]\lra { \R}$ denoting the
curvature. 

\end{problem}
{Note also
$$(\fal \vph\in \C, \vep>0)\quad \EE(\vph)<+\8 \te{ implies that } \vph \text{ is } C^1 \text{ regular, } $$
\grn since by Sobolev inequality in one dimension $H^2$ embeds in $C^1$. 
The term $\int_0^{L(\vph)} \left|\kappa_\vph\right|^2\d\L^1$
will be often referred to as {\em ``curvature term''}. }
For future reference, given a curve $\psi\in \C$,
the notation $\kappa_\psi$ will denote the curvature 
of $\psi$.
Existence of minimizers will be proven in Theorem \ref{ext}. 

The integrated squared curvature penalization is related
to the Willmore energy, introduced in differential geometry and widely used in image analysis.
Among the vast literature regarding Willmore energy, 
 we cite Bretin, Lachuad and Oudet \cite{BLO}, {\color{black} where, the authors used the Willmore energy to develop an algorithm
 	to reconstruct the contours of digital data}, 
Du, Liu, Ryham and Wang \cite{DLRW}, Dondl, Mugnai and R\"oger \cite{DMR},
{\color{black}which are dedicated to formulating an approximation of the Willmore energy via the more
	easily
	 treated
	phase field functions.
}


We note that curvature penalization related to principal curve problem has been considered by Tibshirani \cite{Add4}. 
 Indeed the functional studied in \cite{Add4} has strong connections to ours. The focus of that paper however was on modeling and numerical algorithms. 
 
We should point out that the curvature penalty was actually already  introduced in the elastica functional, which is one of the earliest  functionals of the calculus of variations.  \grn
The elastica problem was proposed in 1691 by Jacob Bernoulli, who studied
``the bendings or curvatures of beams, drawn bows, or of springs of any kind, caused
by their own weight or by an attached weight or by any other compressing forces....'', \cite{Goss}, page 17. 
Indeed, if the measure $\mu$ in Problem \ref{main} is purely atomic, then the term
$\int_{\R^d} d(x,\ggm_\vph)^\p  \d\mu $ is a type of ``pulling'' force on the minimizer,
while length (resp. curvature) term $L(\gm)$ (resp. $\int_0^{L(\vph)} |\kappa_\vph|^2 \d\L^1$)
account for the elastic forces.
The elastica problem was  studied by Bernoulli \cite{Bernoulli1694} who provided a partial solution.
Subsequently, Euler described in \cite{Euler} two radically different solutions. Both of these works played a fundamental role in the history of the calculus of variations and their influence extends to this day. For 
a comprehensive look at the history of elastica problem see the article of Levien \cite{Levien} and the doctoral thesis of Goss \cite{Goss}. 
\nc

 \nc

Problem \ref{main} is also related to the notion of splines \cite{deBoor}.  In particular the cubic smoothing spline for the given collection of data points $\{(x_i,y_i) \::\: x_i \in [0,1],\; y_i\in \R\}$ where $x_i < x_{i+1}$ for all $i=1,\dots, n-1$
is the function $f:[0,1] \to \R$ minimizing 
\[ \frac{1}{n} \sum_{i=1}^n (y_i - f(x_i))^2 +  \vep \int_0^1 f''(x) dx. \]
Between the points $x_i$ and $x_{i+1}$ the function $f$ is a cubic polynomial, while at the points $x_i$, $f$ is differentiable, but the second derivative typically has a jump. 
As in our problem $\vep$ controls the level of smoothness. In the $\vep \to \infty$ limit, the problem converges to the best linear fit, while in the $\vep \to 0+$ limit the solutions converge to the interpolating cubic spline. 

The main difference compared to our problem  is that if one considers splines as curves in 2D, $x \mapsto (x, f(x))$ then one coordinate, $x$ is given a-priory, and it already specifies the order in which the points are visited. The lack of information/requirement about the order in which the data points are to be approximated is a major source of complication for our problem. Among other things it is the reason why one needs to add the length penalty, as otherwise the problem does not have a minimizer. 
\nc 
\medskip

The connection between Problems \ref{main} and \ref{main1}
will be discussed in Lemma \ref{gamma}.
The main goal is to prove that 
minimizers are $C^{1,1}$ regular. 
The main result is:
\begin{thm}\label{1-1}
Given $d\geq 2$, a finite measure $\mu$ \grn with compact support in $\R^d$, \nc and parameters $\la,\vep>0$, $\p\geq1$, any minimizer $\vph$ of $\EE$
 is $C^{1,1}$ regular and
$\vph'$ is 
Lipschitz continuous, with Lipschitz constant not greater than
\begin{equation}\label{Y}
Y\big(\!\D,\la,\p,\vep\big):=\sqrt{ \frac{2(2\vartheta+\la)}{\vep} },
 \end{equation}
where 
\begin{equation}\label{cal2}
\vartheta:=\p D^{\p-1},\qquad D:=2\D .
\end{equation}
\end{thm}

Note that
$Y\big(\!\D,\la,\p,\vep\big)$ 
diverges as 
$\vep\ra 0$: this is indeed necessary, in view of Lemmas \ref{gamma} and \ref{non reg} below.
This paper is structured as follows.
In Section \ref{prel} we present some preliminary definitions 
and results, and prove existence of minimizers
of Problem \ref{main}. 
In Section \ref{reg} we prove that minimizers of Problem \ref{main} are $C^{1,1}$ regular.
Moreover, we obtain estimates on  the mass projecting onto a subset in relation
to its length.

\section{Preliminaries} \label{prel}
The aim of this section is to present preliminary notions and results, which will be useful in Section \ref{reg}. 
In particular we will prove existence of minimizers for Problem \ref{main} in Theorem \ref{ext}. 
The first step is to endow the space $\C$ with a suitable topology. 


\grn
We define the following metric on $\C$: 
 \begin{equation} \label{def:metC}
m(\vph,\psi)=\red \min \left\{ \|\Phi(\vph)-\Phi(\psi) \|_{L^\8}, \|\Phi(\vph(L(\vph) - \cdot)-\Phi(\psi) \|_{L^\8} \right\}  \nc +{     \big|	 L(\vph)- L(\psi) 	\big|}.    \nc
\end{equation}
It is straightforward to check that $m$ is a metric on $\C$. Intuitively, the convergence with respect to $m$ is equivalent to \red orientation invariant \nc uniform convergence, 
plus convergence of length.
%
%
In proofs it is sometimes useful to identify the preferred orientation along a sequence. In particular  given a sequence $\{\vph_n\}\sse\C$ and  $\vph\in \C$ we let
\[\zeta_n :=\begin{cases}
	\Phi(\vph_n) & \text{if } \| \Phi(\vph_n) -\Phi( \vph)  \|_{L^\8(0,1)} \le 
	 \| \Phi(\vph_n(L(\vph_n)-\cdot)) -\Phi( \vph)  \|_{L^\8(0,1)} ,\\
	 \Phi(\vph_n(L(\vph_n)-\cdot))& \text{if } \| \Phi(\vph_n) -\Phi( \vph)  \|_{L^\8(0,1)} \ge 
	 \| \Phi(\vph_n(L(\vph_n)-\cdot)) -\Phi( \vph)  \|_{L^\8(0,1)} ,
	\end{cases},\]
where $\Phi$ is the function defined in \eqref{Phi}.
We note that the  sequence $\{\vph_n\}$,  converges to
$\vph$ with respect to metric $m$ and we write $\{\vph_n\}\rca \vph$ if: 

\begin{itemize}
\item[(i)] $L(\vph_n)\ra L(\vph)$ as $n \ra \*$ and
\item[(ii)] {   the sequence
	$\{\zeta_n\}$, 
	converges uniformly to $\Phi( \vph)$ as $n \ra \8$.
	  }
\end{itemize}
\nc

%
%

We note that the notion of convergence implies
that if $\vph_n\rca\vph$, then for arbitrary times $t_1<t_2$ it holds
$L( \Phi^{-1}(\zeta_{n})_{|[t_1,t_2]} ) \to L(\vph_{|[t_1,t_2]})$ as $n \to \infty$, (i.e., the lengths of restrictions converge).

\medskip

We turn to establishing some technical lemmas. {  Note that for any measure $\mu$, $\la>0$, $\p\geq 1$,
it holds
\begin{equation}\label{EE finite}
\inf_\C E_\mu^{\la,p}\leq \int_{\supp(\mu)} |x-z|^p \d\mu(x)\leq \diam(\supp(\mu))^p<+\8,
\end{equation}
where 
\begin{equation*}
 E_\mu^{\la,p}(\vph):=\int_{\R^d} d(x,\ggm_\vph)^p\d\mu(x) +\la L(\vph) ,
\end{equation*}
and
$z\in \supp(\mu)$ is an arbitrarily chosen point, since clearly $\psi: { \{0\}} \to
\{z\}$ belongs to $\C$. In particular, it can be assumed that any minimizing sequence $\{\gm_n\}$
satisfies
$$\sup_n L(\gm_n)<+\8,\qquad 
\sup_n \int_0^{L(\gm_n)} \kappa_{\gm_n}^2 \d s
<+\8. $$}

\begin{lem}\label{conv}
Given $d\geq 2$, a compact set $K\sse\R^d$, let
$\{\gm_n\}:[0,1]\lra K$ be a sequence of constant speed
parameterized curves satisfying
$$\sup_n \|\gm_n'\|_{TV}<+\8,\qquad \sup_n L(\gm_n)<+\8,$$
with $\|\cdot\|_{TV}$ denoting the total variation semi-norm.
Then \grn there exists a subsequence of $\{\gm_n\}$ which we do not relabel, \nc and 
 a curve $\gm:[0,1]\lra K$ such that
$$\begin{array}{ll}
\{\gm_n\}\ra \gm & \text{in } C^\al \text{ for any } \al\in [0,1),\\
\{\gm_n'\}\ra \gm' & \text{in } L^p \text{ for any } p\in [1,+\8),\\
\{\kappa_{\gm_n}\}\rhu \kappa_\gm & \text{in the space of signed Borel measures}.
\end{array}$$
\end{lem}

{  It is worth noting that in our case, since we work with
curves $\gm$ with uniformly bounded energy, and hence
uniformly bounded $\int_0^{1} \kappa_{\gm}^2 \d s$,
the weak-* convergence of the curvature measures implies the weak $L^2$
convergence. }

\begin{proof}
For the proof we refer to \cite[Lemma~6]{Slepcev}.
\end{proof}

{ 

\begin{lem}\label{length C0}
Given a sequence of constant (positive) speed curves
 $\{\gm_n\}:[0,1]\to\R^d$, converging uniformly to $\gm:[0,1]\to \R^d$, 
such that 
\begin{equation}\label{finiteness}
\sup_{n}|\gm_n'|<+\8,\qquad \sup_n \int_0^{1} \kappa_{\gm_n}^2 \d s<+\8,
\end{equation}
then it holds
$$L(\gm)=\lim_{n\to +\8} L(\gm_n).$$
\end{lem}

Note that this is a much stronger result { than} the general lower semicontinuity of
length. In particular, due to the curvature penalization, it states
that any minimizing sequence $\{\gm_n\}$ (which surely satisfies \eqref{finiteness}
which admits a uniform limit
$\gm$, satisfies also $\{L(\gm_n)\}\to L(\gm)$. This will be crucial for the proof of Lemma
\ref{C closed}.

{  
For brevity, given two vectors $v_1,v_2$, the notation
$\angle v_1 v_2$ will denote the angle between $v_1$ and $v_2$. That is,
$$\angle v_1 v_2 \in [0,\pi],\qquad \angle v_1 v_2:=\arccos\frac{\langle v_1, v_2
	\rangle}{|v_1||v_2|},$$
where $\langle \cdot,\cdot\rangle$ denotes the standard Euclidean scalar product 
of $\R^d$.}

\begin{proof}
Let $\{\gm_n\}$ be a sequence satisfying \eqref{finiteness}. 
\medskip

\noindent Case $1^o$: $\ggm_\gm$ is a singleton. 
{  We show that 
	since $\{\gm_n\}\to \gm$ uniformly, the uniform bound on the integrated squared curvature gives
	$\{L(\gm_n)\}\to 0$.
	Without loss of generality, let $\ggm_\gm=\{0\}$, and let
	 $\{\vep_n\}\da 0$ be a sequence such that
	$\ggm_{\gm_n} \sse B(0,\vep_n)$. 
	Let $\~\gm_n$ be the arc-length reparameterizations of $\gm_n$.
	We first claim that:
\begin{itemize}
	\item for any $n$, and a.e. $t\le L(\gm_n)-3\sqrt{2}\vep_n$, 
	there exists $s\in [t,t+3\sqrt{2}\vep_n]$ such that
	\[  \langle\~\gm_n'(t), \~\gm_n'(s)\rangle \le \frac{1}{\sqrt{2}}. \]
\end{itemize} 
To show this claim, it suffices to note that if this were not the case,
i.e. $\langle\~\gm_n'(t), \~\gm_n'(s)\rangle \ge \frac{1}{\sqrt{2}}$ for all  
$s\in [t,t+3\sqrt{2}\vep_n]$, 
 then 
\[  \langle\~\gm_n(t+3\sqrt{2}\vep_n) -\~\gm_n(t),\gm_n'(t)  \rangle  \ge \int_t^{t+3\sqrt{2}\vep_n}
 \langle
\~\gm_n(s),\gm_n'(t)  \rangle\d s \ge 3\vep_n, \]
which contradicts the fact that $\ggm_{\gm_n} \sse B(0,\vep_n)$. The claim is thus proven.

If $L(\gm_n)> 3\sqrt{2}\vep_n$ were to hold, then there exist $s_n> t_n$ such that
$$s_n- t_n\le 3\sqrt{2}\vep_n,\qquad \langle\~\gm_n'(t_n), 
\~\gm_n'(s_n)\rangle \le \frac{1}{\sqrt{2}}.$$
 This in turn gives
\begin{align*}
 \int_0^{1} \kappa_{\gm_n}^2 \d s\ge \int_{t_n}^{s_{n}} |\~\gm_n''|^2\d s\ge 
 \frac{1}{s_{n}-t_n}\bigg|\int_{t_n}^{s_{n}} \~\gm_n''\d s\bigg| 
\ge \frac{ |\~\gm_n'(t_n)-\~\gm_n'(s_n)|^2 }{3\sqrt{2}\vep_n}
\ge \frac{|1-\frac{1}{\sqrt{2}}|^2}{3\sqrt{2}\vep_n},
\end{align*}
where the last inequality follows from
\[ |\~\gm_n'(t_n)-\~\gm_n'(s_n)| \ge \langle \~\gm_n'(t_n)-\~\gm_n'(s_n), \~\gm_n'(t_n) \rangle
= 1 -\underbrace{\langle \~\gm_n'(s_n), \~\gm_n'(t_n) \rangle}_{\le 1/\sqrt{2}}. \]
 This clearly contradicts \eqref{finiteness}. Thus the only possibility 
 is  $L(\gm_n)\le 3\sqrt{2}\vep_n$ for all sufficiently large $n$, which clearly
 implies $L(\gm_n)\to 0=L(\gm)$.
}

\medskip

\noindent Case $2^o$ $\ggm_\gm$ is not a singleton. Without loss of generality, 
	{ we can} assume
$\inf_n |\gm_n'|>0$. {  This because by assumption 
	we have $\gamma_n\to \gamma$ uniformly.
	As $\gamma_n$ are arc-length parameterized, $|\gamma_n'|$ is equal to its length
	$L(\gm_n)=\mathcal{H}^1(\gamma_n)$, with $\mathcal{H}^1$ denoting the Hausdorff-1 measure.
	By the lower semicontinuity of 
	$\mathcal{H}^1$ (Golab's theorem) we then have
	$\liminf_{n\to+\8} \mathcal{H}^1(\gamma_n)\ge \mathcal{H}^1(\gamma)>0$. }

Note:
\begin{enumerate}
\item  $\{\gm_n\}\to \gm$ uniformly, i.e. $\{\gm_n\}$
is bounded in $L^\8([0,1];\R^d)$;

\item since the speed of $\{\gm_n\}$ is equal to $|\gm_n'|$ a.e.,
hypothesis $\sup_{n}|\gm_n'|<+\8$ gives that $\{\gm_n\}$ is bounded in $W^{1,\8}([0,1];\R^d)$;

\item the curvature $\kappa_{\gm_n}$ describes the variation of the unit tangent derivative,
and absolute continuity of $\kappa_{\gm_n}$ with respect to the Lebesgue measure (which follows from
\eqref{finiteness}) gives
$$\kappa_{\gm_n}= |\gm_n'|^{-1}\frac\d{\d s}\Big( \frac{\gm_n'}{|\gm_n'|}\Big)=
|\gm_n'|^{-1}\frac{\gm_n''}{|\gm_n'|}.$$
Since $\inf_{n} |\gm_n'|>0$, \eqref{finiteness} implies that
 $\{\gm_n\}$ is bounded in $W^{2,2}([0,1];\R^d)$.
\end{enumerate}
Thus there exists $\~\gm\in W^{2,2}([0,1];\R^d)$ and a subsequence $\{\gm_{n_k}\}$ 
converging to $\~\gm$ in the weak topology of $W^{2,2}([0,1];\R^d)$. Since the embedding
$W^{2,2}([0,1];\R^d) \hookrightarrow W^{1,2}([0,1];\R^d)$ is compact,
$\{\gm_{n_k}\}$ converges to $\~\gm$ strongly in $W^{1,2}([0,1];\R^d)$,
hence $\{\gm_{n_k}\}$ converges to $\~\gm$ strongly in $L^2([0,1];\R^d)$.
Since by hypothesis $\{\gm_{n_k}\}$ converges to $\gm$ strongly in $L^\8([0,1];\R^d)$,
it follows $\~\gm=\gm$. Strong convergence of $\{\gm_{n_k}\}$ to $\gm$
in $W^{1,2}([0,1];\R^d)$ implies 
$$L(\gm_{n_k})=\int_0^1 |\gm_{n_k}'|\d s   \to \int_0^1 |\gm'|\d s = L(\gm).$$
{ 
To achieve convergence for the whole sequence, note that if there exists
a subsequence $\{\gm_{n_h}\}$ such that $\lim_{h\to+\8}  L(\gm_{n_h})\neq L(\gm)$,
then the above construction allows to further extract a sub-subsequence
 $\{\gm_{n_{h(j)}}\}$ converging to $\gm$ in the strong topology of $W^{1,2}([0,1];\R^d)$,
which in turn implies $\lim_{j\to+\8} L(\gm_{n_{h(j)}})= L(\gm)$.
Contradiction.
}
\end{proof}

The next result proves compactness of sub-levels for $\E_\mu^{\la,\p}$.
}

\begin{lem}\label{C closed}
Given $d\geq 2$, {  finite measure $\mu \geq 0$ compactly supported in $\R^d$},  parameters $\la,\vep>0$, $\p\geq 1$,
$M \geq \inf_\C \EE$, 
a sequence $\{\vph_n\}\sse \C\cap \{\EE\leq M\}$, there
exists $\vph_\8\in \C$ such that (upon subsequence, which will not be relabeled) $\{\vph_n\}\rca \vph_\8$.
\end{lem}

\begin{proof}
Consider an arbitrary sequence $\{\vph_n\}\sse\C\cap \{\EE\leq M\}$.
If there exists a subsequence $\{n_k\}$ such that
$\{L(\vph_{n_k})\}\ra 0$, then $\{\ggm_{\vph_{n_k}}\}\rdh \{z\}$ for some $z\in \R^d$,
and letting $\vph_\8: { \{0\}}\lra \{z\}$ concludes the proof. Now assume
\begin{equation}\label{assum cpt}
 \liminf_{n\to+\8} \nc L(\vph_n)>0.
\end{equation}
 Let $\vph_n^*:=\Phi(\vph_n)$, $n\in \mathbb{N}$. {  We show that $\{\vph_n^*\}$ is bounded in 
 $W^{2,2}([0,1];\R^d)$.}

\medskip

Claim 1. $\sup_n  L(\vph_n^*)<+\8.$

Consider an arbitrary index $n$. Since
$\EE(\vph_n)\leq M$,
 it follows
 \begin{equation*}
 \la L(\vph_n^*)= \la L(\vph_n) \leq \EE(\vph_n) \leq M   \Lra L(\vph_n^*)\leq M/\la.
 \end{equation*}
 
 \medskip

 Claim 2. { ${\disp{\sup_n \int_0^{1} \left| \kappa_{\vph_n^*} \right|^2 \d\L^1<+\8}}$}.

Consider an arbitrary index $n$. Note that, for any $\vph\in \C\cap \{\EE\leq M\}$ 
{  with $L(\vph)>0$}, it holds
 \begin{align}
\int_0^1 |\kappa_{\vph^*}|^2 \d\L^1 & = \int_0^1 \left| \frac{\d}{\d s} 
\left( \frac{\vph'\big(sL(\vph)\big)}{|(\vph^*)'|} \right) \right|^2 \d s 
 =  \int_0^1 \left| \vph''\big(sL(\vph)\big)  \right|^2 \d s \notag \\
 &
\overset{s':=sL(\vph)}{ =} \frac{1 }{L({\vph}) } \int_0^{L(\vph)} \left| \vph''(s') \right|^2 \d s'
= \frac{1 }{L({\vph}) } \int_0^{L(\vph)} \left| \kappa_\vph \right|^2 \d\L^1 ,\label{note}
\end{align}
where $\vph^*:=\Phi(\vph)$.
Condition $\EE(\vph_n) \leq M$ then gives
 \begin{equation*}
 \vep L({\vph_n}) \int_0^1 |\kappa_{\vph_n^*}|^2 \d\L^1 = \vep 
  \int_0^{L(\vph_n)} \left| \kappa_\vph \right|^2 \d\L^1 \leq
 \EE(\vph_n) \leq M. 
 \end{equation*}

\medskip

 Claim 3. $\bigcup_n \ggm_{\vph_n} \sse K$ for some compact set $K$.

Consider an arbitrary index $n$. Given $r$, if $\ggm_{\vph_n}\cap \big(\!\supp(\mu)\big)_r =\emptyset$,
where 
$$\big(\!\supp(\mu)\big)_r:=\left\{x\in \R^d: d\big(x,\supp(\mu)\big)\leq r\right\},$$ 
then
$$\int_{\R^d} d(x,\ggm_{\vph_n})^\p\d\mu \geq r^\p.$$
Claim 1. proves $\sup_n L(\vph_n)\leq M/\la$, thus if 
$\ggm_{\vph_n}\cap \left(\R^d\backslash \big(\!\supp(\mu)\big)_{ (2M)^{1/\p}+ 
M/\la} \right) $ is non empty, then
$$\ggm_{\vph_n}\cap \big( \supp(\mu)\big)_{(1.5M)^{1/\p}}=\emptyset,$$  hence
$$M\geq \EE(\vph_n)\geq \int_{\R^d} d(x,\ggm_{\vph_n})^\p\d\mu \geq 1.5M. $$
Contradiction. Letting $K:=\big(\!\supp(\mu)\big)_{(2M)^{1/\p}+ 
M/\la}$
proves the claim.

\medskip

{  Claims 1., 2. and 3. prove that $\{\vph_n^*\}$ is bounded in $W^{2,2}([0,1];\R^d)$,
hence (upon subsequence) it converges to $\vph_\8^*$ in the strong topology of $W^{1,2}([0,1];\R^d)$.
{  Therefore, denoting by $\vph_\8$ the arc-length reparameterization of $\vph^*$,
we get $\{\vph_n\}\overset{\C}{\to} \vph_\8$, concluding the proof.}
}
\end{proof}

Now we can prove existence of minimizers.
\begin{thm}\label{ext}
 For $d\geq 2$, given {  finite measure $\mu \geq 0$ compactly supported in $\R^d$},  and parameters $\la,\vep>0$,
 $\p\geq 1$,
 the functional $\EE$ admits a minimizer in $\C$.
\end{thm}

\begin{proof}
Consider an arbitrary minimizing sequence $\{\vph_n\}\sse \C$.
In view of \eqref{EE finite}, assume 
$$\EE(\vph_n)\leq \big(\!\D\big)^\p+\xi \qquad\text{for some } {  0<}\xi\ll1.$$
Lemma \ref{C closed} gives the existence of a limit curve $\vph_\8$.
Convergence $\{\vph_n\}\rca \vph_\8$ gives 
\begin{equation}\label{conv conq}
\lim_{n\to+\8} \int_{\R^d} d(x,\ggm_{\vph_n})^\p\d\mu =  
\int_{\R^d} d(x,\ggm_{\vph_\8})^\p\d\mu,\qquad \lim_{n\to+\8}L(\vph_n)=L(\vph_\8).
\end{equation}
It remains to prove
\begin{equation}\label{goal}
\liminf_{n\to+\8} \int_0^{L(\vph_n)} |\kappa_{\vph_n}|^2\d\L^1 \geq 
\int_0^{L(\vph_\8)} |\kappa_{\vph_\8}|^2\d\L^1. 
\end{equation}
Let $\vph_n^*:=\Phi(\vph_n)$, $n\in \mathbb{N}$, and
$\vph_\8^*:=\Phi(\vph_\8)$.
Note that
$$\sup_n \int_0^{L(\vph_n)} |\kappa_{\vph_n}|^2\d\L^1 
\leq \sup_n \frac{\EE(\vph_n)}\vep \leq \frac{\big(\!\D\big)^\p+\xi}\vep<+\8,$$
thus, in view of \eqref{note},
$\sup_n \int_0^{L(\vph_n)} |\kappa_{\vph_n^*}|^2\d\L^1 <+\8$,
and
 the sequence $\{\d\kappa_{\vph_n}/\d\L^1\}$
is bounded in $L^2([0,1])$. Therefore
(upon subsequence, which will not be relabeled) there exist
$g\in L^2([0,1])$ such that
$\{\d\kappa_{\vph_n^*}/\d\L^1\}\rightharpoonup g$, hence 
\begin{equation}\label{curve lsc}
\left\{ \frac{\d\kappa_{\vph_n^*}}{\d\L^1}\cdot \L^1_{\llcorner [0,1]} \right\}
\rightharpoonup g\cdot \L^1_{\llcorner [0,1]} \quad { \text{ weakly in }
L^2([0,1])}, 
\qquad
\|g\|_{L^2([0,1])}\leq \liminf_{n\to+\8} \left\| \frac{\d\kappa_{\vph_n^*}}{\d\L^1} \right\|_{L^2([0,1])}.
\end{equation}
Lemma \ref{conv} gives (upon subsequence, which will not be relabeled)
 $\{\kappa_{\vph_n^*}\}\rhu \kappa_{\vph_\8^*}$, 
hence 
$\kappa_{\vph_\8^*}=g\cdot \L^1_{\llcorner [0,1]}$. Combining
\eqref{curve lsc}, observation \eqref{note}, and $\{L(\vph_n)\}\ra L(\vph_\8)$
proves \eqref{goal}. Combining \eqref{conv conq} and \eqref{goal}
gives
$\EE(\vph_\8) \leq \liminf_{n\to+\8} \EE(\vph_n)$, concluding the proof. 
\end{proof}

\begin{lem}\label{comp}
Given $d\geq 2$, {  a finite measure $\mu \geq 0$ compactly supported in $\R^d$},  parameters $\la,\vep>0$, 
$\p\geq 1$ and
a minimizer $\vph\in \argmin\ \E_\mu^{\la,\vep}$, it holds:
\begin{listi}
\item length estimate: 
\begin{equation}\label{e1}
L(\vph)\leq \frac{\big(\!\D\big)^\p}{\la},
\end{equation}

\item curvature term estimate:
\begin{equation}\label{e2}
\int_0^{L(\vph)} |\kappa_\vph|^2\d\L^1 \leq \frac{\big(\!\diam \supp(\mu)\big)^\p}{\vep},
\end{equation}

\item confinement condition: $\ggm_\vph \sse \big(\!\supp(\mu)\big)_{\D+
(\D)^{\p}/\la}$,
where for given $r\geq 0$, 
$$\big(\!\supp(\mu)\big)_r:=\{x\in \R^d:d\big(x,\supp(\mu)\big)\leq r\}.$$
\end{listi}
\end{lem}

\begin{proof}
Estimates \eqref{e1} follows from 
$$\la L(\vph)\leq \EE(\vph) = \inf_\C \EE \overset{\eqref{EE finite}}{\leq}
\big(\!\diam \supp(\mu)\big)^\p, $$ 
while \eqref{e2} follows from
$$\vep \int_0^{L(\vph)} |\kappa_\vph|^2
\d\L^1\leq \EE(\vph) = \inf_\C \EE \overset{\eqref{EE finite}}{\leq}
\big(\!\diam \supp(\mu)\big)^\p. $$
To prove the confinement condition, note
that for any $\psi\in \C$, if $\ggm_\psi\cap \big(\!\supp(\mu)\big)_r =\emptyset$,
then it follows $\int_{\R^d} d(x,\ggm_\psi)^\p\d\mu\geq r^\p$. Inequality
\eqref{EE finite} gives $\ggm_\psi\cap \big(\!\supp(\mu)\big)_{\D} \neq\emptyset$ .
Combining with length estimate \eqref{e1} concludes the proof.
\end{proof}

\medskip

The next result proves a connection between Problems \ref{main} and \ref{main1}.
Let
$E_\mu^{\la,\p}:\C\lra [0,+\8)$ be defined by 
$$   E_\mu^{\la,\p}(\psi):=\int_{\R^d} d(x,\ggm_\psi)^\p\d\mu+\la L(\psi).$$
\begin{lem}\label{gamma}
Given $d\geq 2$, {  a finite measure $\mu \geq 0$ compactly supported in $\R^d$},  parameters $\la>0$,
$\p\geq 1$,
and a sequence $\{\vep_n\}\ra 0$, then
 $\{\E_{\mu}^{\la,\vep_n,\p}\}$ {  $\ggm$-converges to} $ E_\mu^{\la,\p}$ { in the topology of $\mathcal{C}$}
 as $n\ra +\8$.
\end{lem}

\begin{proof}
The proof of $\ggm$-convergence is be split into two steps.
\medskip

Step 1. $\ggm-\liminf$ inequality: for any sequence $\{\vph_n\}\rca \vph$
it holds $$\liminf_{n\to+\8}\E_{\mu}^{\la,\vep_n,\p}(\vph_n)\geq E_\mu^{\la,\p}(\vph).$$
Since $\{\vph_n\}\rca \vph$,
it follows $\{\ggm_{\vph_n}\}\rdh \ggm_\vph$ and $\{L(\vph_n)\}\ra L(\vph)$. As proven in \cite{BOS},
the functional
$$ \A\ni \sm \mapsto \int_{\R^d} d(x,\sm)^\p \d\mu $$
is continuous with respect to Hausdorff distance, hence
$$\lim_{n\to+\8}  \int_{\R^d} d(x,\ggm_{\vph_n})^\p\d\mu= \int_{\R^d} d(x,\ggm_{\vph})^\p\d\mu.$$
Thus we have
\begin{align*}
\liminf_{n\to+\8} \E_\mu^{\la,\vep_n,\p}(\vph_n) & \geq \liminf_{n\to+\8} \int_{\R^d} d(x,\ggm_{\vph_n})^\p\d\mu +\la L(\vph_n) \\
& = \int_{\R^d} d(x,\ggm_{\vph})^\p\d\mu+\la L(\vph)\\
& = E_\mu^{\la,p}(\vph).
\end{align*}

\medskip

Step 2.  $\ggm-\limsup$ inequality: for any $\vph$, there exists a sequence $\{\vph_n\}\rca \vph$
such that 
$${ \limsup_{n\to+\8}}\ \E_{\mu}^{\la,\vep_n,\p}(\vph_n)\leq E_\mu^{\la,\p}(\vph).$$

{  Let $\{ \tilde \vph_k\}$ be a sequence of smooth functions such that $ \{\tilde \vph_k\}\to \vph$ uniformly.  By relabeling the sequence, using $k(n) = \max \left\{1, \sup \left\{ k \::\: 
  \int_0^{L(\tilde \vph_k)} |\kappa_{\tilde \vph_k}|^2 \d\L^1 \leq \vep_n^{-1/2} \right\} \right\}$, and defining
  $\vph_n = \tilde \vph_{k(n)}$ we have that for all $n$ large enough
$$
\int_0^{L(\vph_n)} |\kappa_{\vph_n}|^2 \d\L^1 \leq \vep_n^{-1/2},\quad \te{and } \;\;
\{\vph_n\}\to \vph \quad \text{uniformly}. $$
}
Thus 
$$\lim_{n\to+\8}\int_{\R^d} d(x,\ggm_{\vph_n})^\p\d\mu = \int_{\R^d} d(x,\ggm_{\vph})^\p\d\mu,
\qquad \lim_{n\to+\8} \vep_n \int_0^{L(\vph_n)} |\kappa_{\vph_n}|^2 \d\L^1 =0,$$
{ which gives}
\begin{align*}
\lim_{n\to+\8} \E_\mu^{\la,\vep_n,\p}(\vph_n) & = \lim_{n\to+\8} \int_{\R^d} d(x,\ggm_{\vph_n})^\p\d\mu 
+\la L(\vph_n) +\vep_n \int_0^{L(\vph_n)} |\kappa_{\vph_n}|^2 \d\L^1\\
& = \int_{\R^d} d(x,\ggm_{\vph})^\p\d\mu+\la L(\vph)\\
& = E_\mu^{\la,\p}(\vph),
\end{align*}
concluding the proof.
\end{proof}

Note that the upper bound in \eqref{e2} diverges as $\vep\ra 0$: this is indeed necessary,
in view of Lemmas \ref{gamma}, and \ref{non reg} below:

\begin{lem} \label{non reg}
(\cite{Slepcev})
There exist $\mu$ and $\la$ such that $E_\mu^{\la,1}$ {  admits a unique minimizer
which is not} $C^1$ regular.
\end{lem}


\begin{lem}
Given $d\geq 2$, a sequence of measures $\{\mu_n\}$ compactly supported in $\R^d$, a {  finite measure $\mu \geq 0$ compactly supported in $\R^d$},  such that $\{\mu_n\}
\rhu\mu $, and parameters $\la,\vep>0$, $\p\geq 1$, then 
$\{\E_{\mu_n}^{\la,\vep,\p}\}$ {  $\ggm$-converges to} $ \EE$ {  in the
topology of $\mathcal{C}$}.
\end{lem}

\begin{proof}
The proof will be split in two steps.

\medskip

 $\ggm-\liminf$ inequality: for any sequence $\{\vph_n\}\rca \vph$,
it holds $\liminf_{n\to+\8} \E_{\mu_n}^{\la,\vep,\p}(\vph_n) \geq \EE(\vph)$.

If $\liminf_{n\to+\8} \E_{\mu_n}^{\la,\vep,\p}(\vph_n)=+\8$, then the $\ggm-\liminf$ inequality follows.
Thus assume (upon subsequence, which will not be relabeled)
$\liminf_{n\to+\8} \E_{\mu_n}^{\la,\vep,\p}=\lim_{n\to+\8} \E_{\mu_n}^{\la,\vep,\p}(\vph_n)<+\8$.
Convergences $\{\mu_n\}\rhu\mu$ and $\{\vph_n\}\rca\vph$
give $\{L({\vph_n})\}\ra L(\vph)$, and
\begin{align} 
\lim_{n\to+\8}  \left|
 \int_{\R^d} d(x,\ggm_{\vph_n})^\p \d\mu_n - \int_{\R^d} d(x,\ggm_{\vph})^\p \d\mu \right|
&= \lim_{n\to+\8} \left|\int_{\R^d} d(x,\ggm_{\vph_n})^\p \d\mu_n -  \int_{\R^d} d(x,\ggm_{\vph_n})^\p \d\mu \right|
\notag\\ 
&+ \left|\int_{\R^d} d(x,\ggm_{\vph_n})^\p \d\mu -  \int_{\R^d} d(x,\ggm_{\vph})^\p \d\mu \right| 
=0.\label{gamma 1}
\end{align}
Then { the} same arguments from the proof of Theorem \ref{ext}
prove lower semicontinuity for the curvature term, hence
${\disp{\liminf_{n\to+\8} \E_{\mu_n}^{\la,\vep}(\vph_n)\geq \EE(\vph)}}$.

\medskip

 $\ggm-\limsup$ inequality: for any $\vph\in \C$,
there exists a sequence $\{\vph_n\}\rca \vph$
such that $${ \limsup_{n\to+\8}}\ \E_{\mu_n}^{\la,\vep,\p}(\vph_n) \leq \EE(\vph).$$

Given $\vph\in \C$, 
case $\EE(\vph)=+\8$ is trivial. Assume $\EE(\vph)<+\8$, and
let $\vph_n:=\vph$ for any $n$. Clearly
$$(\fal n) \quad L({\vph_n})= L(\vph), \qquad 
\int_{0}^{L({\vph_n})} |\kappa_{\vph_n}|^2 \d\L^1 = \int_{0}^{L(\vph)} |\kappa_{\vph}|^2 \d\L^1. $$
Since $\{\mu_n\}\rhu \mu$, it follows
$$\lim_{n\to+\8}\int_{\R^d}d(x,\ggm_{\vph_n})^\p\d\mu_n =\int_{\R^d}d(x,\ggm_{\vph})^\p\d\mu,$$
concluding the proof.
\end{proof}

\section{Regularity} \label{reg}
The aim of this section is to prove Theorem \ref{1-1}{ .}
As a consequence, we have Corollaries \ref{points 0}
and \ref{mass proj}, which estimate (for minimizers) the mass projecting on a subset
in relation to its length. This proves that Problem \ref{main}
is effectively a better candidate for application in data parameterization
than Problem \ref{main1}, since it does not exhibit the undesirable
properties described in ($*$) and ($**$). The results are proven for 
generic (finite) dimension. However
while
 Theorem \ref{1-1} is proven for generic measures,
Corollaries \ref{points 0} and \ref{mass proj} are restricted to
 absolutely continuous measures.

\medskip

We first prove a weaker regularity result.

\begin{lem}\label{half}
Given $d\geq 2$, a {  finite measure $\mu \geq 0$ compactly supported in $\R^d$}, and parameters $\la,\vep>0$, $\p\geq 1$,
any minimizer $\vph\in \argmin_\C \EE$ is $C^{1,1/2}$ regular with
$$\big(\fal t_0,t_1\in [0,L(\vph)]\big)
\qquad |\vph'(t_1)-\vph'(t_0)| \leq \left(\frac{\big(\!\D\big)^\p}{\vep}|t_1-t_0|\right)^{1/2}.$$
\end{lem}

\begin{proof}
Let $\vph$ be an arbitrary minimizer. 
{ Definition \eqref{curv} requires that $\vph\in H^2([0, L(\vph)];\R^d)$,
which, 
} 
{\color{black} by Sobolev embedding,}
implies that $\vph'$ is continuous.

 Thus
$\vph$ is $ C^1$ regular.
H\"older inequality gives
\begin{align*}
\frac{|\vph'(t_1)-\vph'(t_0)|^2}{|t_1-t_0|}&\leq
\frac1{|t_1-t_0|}\left(\int_{t_0}^{t_1} |\kappa_\vph| \d\L^1\right)^2 \leq
\int_{t_0}^{t_1} |\kappa_\vph|^2 \d\L^1
{ 
\le \frac{\EE(\vph)}{\vep} \le \frac{\big(\!\D\big)^\p}{\vep},
}
\end{align*}
{ 
with the last inequality due to the fact that $\vph$ has no more energy
compared to any singleton $\psi:\{0\}\lra \{z\}$, $z\in \supp(\mu)$, i.e.
\[\EE(\vph) \le \EE(\psi) \le
\int_{\supp(\mu)} |x-z|^p \d\mu(x)\leq \diam(\supp(\mu))^p. \]
}
%
{ The proof is thus complete.}
\end{proof}

Note that the proof of Lemma \ref{half} is quite simple, and
consists of comparing the curvature term with condition \eqref{EE finite},
without any kind of construction. As one may expect,
this kind of proof can be refined, and stronger regularity results 
can be established.

\subsection{Lipschitz regularity of derivatives}

{  Note that the mean value theorem implies that for {  any $0<C< +\infty$ }
 \begin{equation}\label{p lem}
(\fal\p\geq 1, \; 0\leq a,b\leq { C})\qquad |a^p-b^p|\leq 
\p |a-b|{ C}^{\p-1}. \qquad
\end{equation}
}
%


Now we are ready to prove Theorem \ref{1-1}.
Note that,
for an vectors $v_1,v_2\in S^{d-1}$ (unit ball
of $\R^d$) 
it holds that
\begin{equation} \label{1-1 2 lem}
\frac{1}{2}\angle v_1v_2 \leq |v_1-v_2|\leq \angle v_1 v_2.
\end{equation}
This follows from the fact that,
$|v_1-v_2|=2\sin\dfrac{\angle v_1v_2}{2} \geq \angle v_1v_2$,  since $0 \leq \angle v_1v_2 \leq \pi$.
{\color{black} For a schematic representation, see Figure \ref{half angle}.}
 \nc


\begin{figure}[ht]
	\includegraphics[scale=.5]{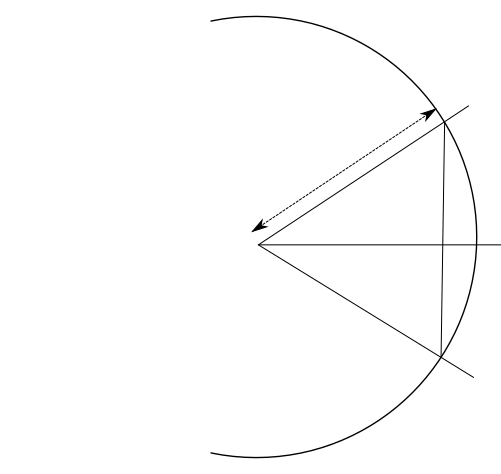}
	\put(-97,80){$0$}
		\put(-28,78){$h$}
	\put(-18,130){$v_1$}
	\put(-23,35){$v_2$}
	\put(-80,88){$\th/2$}
		\put(-62,115){$1$}
\put(-55,165){$S^{d-1}$}
	\caption{Schematic representation of $\th:=\angle v_1v_2$. Elementary geometric
	considerations give $|v_1-v_2|=2|v_1-h| = 2\sin(\th/2) $. }
\label{half angle}
\end{figure}
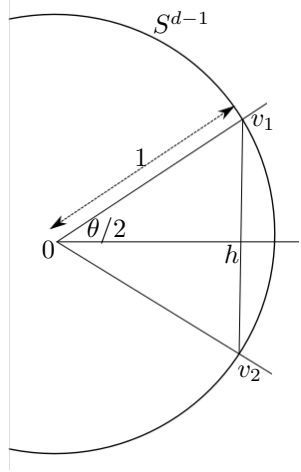

\begin{proof} (of Theorem \ref{1-1})
Let $\vph$ be an arbitrary minimizer. 
Since Lemma \ref{half} proved that $\vph'$ is continuous, 
\begin{equation}\label{1-1 3}
\lim_{|s-t|\ra 0} |\vph'(s)-\vph'(t)| = \lim_{|s-t|\ra 0}\angle \vph'(s)\vph'(t)= 0.
\end{equation}




To show that $\varphi'$ is Lipschitz continuous it suffices to show that 
\[  \sup_{0\leq t_0 < L(\varphi)} \, \limsup_{t_1 \to t_0+}  \frac{| \vph'(t_0)-\vph'(t_1)|}{|t_1-t_0|} < \infty.\]
Consider any $0\leq t_0 < L(\varphi)$ and $M < \limsup_{t_1 \to t_0+}  \frac{| \vph'(t_0)-\vph'(t_1)|}{|t_1-t_0|}$.
\nc

{  The main ideas are:

\begin{listi}
\item Using the original curve $\vph$, we construct the competitor $\psi$ in \eqref{psi defn} below.
Intuitively, $\psi$ is obtained by {  first expanding by a factor of $2$
a piece of the original curve $\vph$ with very high curvature (potentially
losing connectedness in the process), 
and then translating in a suitable way 
the remaining part} to regain connectedness (see \eqref{psi defn}).

\item We estimate the difference $E_\mu^\la(\psi)-E_\mu^\la(\vph)$.
with $E_\mu^\la$ defined in Problem \ref{main1}.  Since the translations 
in the previous step are by vectors with norm at most $\xi$, it follows that
$$\int_{\R^d} [d(x,\ggm_\psi)^p-d(x,\ggm_\vph)^p]\d\mu(x) \le C_1\xi,$$ 
for some geometric constant $C_1>0$
(Claim 1. below). Similarly,
since we will add an arc of circle (to regain connectedness for $\psi$), the difference of lengths
$L(\psi)-L(\vph)$ is exactly $\xi$. Finally,
we estimate the gain for the curvature term (Claim 2.).
\end{listi}
}
Let $h$ be the homothety of center $\vph(t_0)$ and ratio 2. That is,
for any $x\in \R^d$, $h(x)$ is defined \grn by
$h(x) = \vph(t_0) + 2\big(x-\vph(t_0)\big)$. 
For $t_1>t_0$ \nc consider the curve 
 \begin{equation}
\psi:[0,L(\vph)+ \xi]\lra \R^d,\qquad 
\psi(t):=\left\{
\begin{array}{cl} 
\vph(t) & \text{if } t\leq t_0, \\
h\Big(\vph\big( (t_0+ t )/2 \big) \Big) & \text{if } t_0 \leq t \leq t_0+ 2\xi, \\
\vph(t-\xi)+ \big(h(\vph(t_1)) - \vph(t_1) \big) & \text{if } t\geq t_0+ 2\xi,
\end{array}\right.  \label{psi defn}
\end{equation}
\grn
where, 
\[ \xi : =|\vph(t_1)-\vph(t_0)|. \]
 As $|\vph'|\equiv 1$, $\L^1$-a.e.,
and $\psi(t_0)=\vph(t_0)$, it follows that
\begin{equation}\label{1-1 est 2a}
|h(\vph(t_1))-\vph(t_1)| 
=|\vph(t_1)-\vph(t_0)|=\xi.
\end{equation} \nc

\begin{figure}[ht]
	\includegraphics[scale=.8]{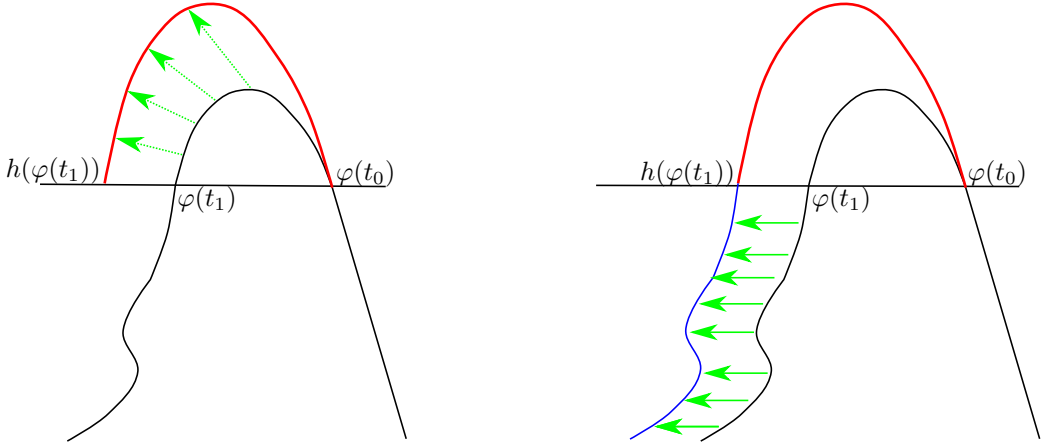}
	\put(-267,100){$\vph(t_0)$}
		\put(-327,90){$\vph(t_1)$}
		\put(-392,101){$h(\vph(t_1))$}
	\put(-28,100){$\vph(t_0)$}
	\put(-87,90){$\vph(t_1)$}
	\put(-152,100){$h(\vph(t_1))$}
	\caption{Left: first, scale the part $\vph([t_0,t_1])$ using the homothety $h$
		with center $\vph(t_0)$ and ratio equal to $2$. 
		Right: translate the part $\vph([t_1,L(\vph)])$ by the vector $h(\vph(t_1))-
		\vph(t_1)$ so to regain connectedness. The competitor $\psi$ is made
		of $\vph([0,t_0])$ from the original curve, and the red and blue pieces.
	}
\end{figure}

By construction $\psi$ is parameterized by arc-length, and
$L(\psi)=L(\vph)+\xi$.
\medskip

 Claim 1. 
$ E_\mu^\la(\psi)-E_\mu^\la(\vph)\leq (2\vartheta +\la )\xi $,
with $\vartheta$ defined in \eqref{cal2}. 
%
Since $\psi_{|[t_1+\xi,L(\vph)+\xi]}$ differs from $\vph_{|[t_1,L(\vph)]}$ by a translation
of the vector $h(\vph(t_1))-\vph(t_1)$, for any point
$$x\in\mathcal{K}_1:=\{z:d(z,\ggm_\vph)=|z-\vph(s)| {\text{ for some }} s\in [t_1,L(\vph)]\}$$
it holds
\begin{align*}
 d(x,\ggm_\psi)\leq  d(x,\ggm_{\psi_{ |[t_1+\xi,L(\vph)+\xi]}})&\leq
 d(x,
\ggm_{\vph_{|[t_1,L(\vph)]}}) + |h(\vph(t_1))-\vph(t_1)|\\
&\overset{\eqref{1-1 est 2a}}{\leq}  d(x,
\ggm_{\vph_{|[t_1,L(\vph)]}}) +  \xi.
\end{align*}
Integrating with respect to $\mu$ gives, in view of \eqref{p lem},
\begin{align}
  \int_{\mathcal{K}_1} d(x,\ggm_{\psi} )^\p \d\mu
- \int_{\mathcal{K}_1} d(x,\ggm_{\vph} )^\p \d\mu  
& \leq \int_{\mathcal{K}_1} |d(x,\ggm_{\psi} )^\p-d(x,\ggm_{\vph} )^\p|\d\mu\notag\\
& \leq 
 \xi {\p \big(D+ \xi\big)^{\p-1}} \mu(\mathcal{K}_1).\label{1-1 est 3}
\end{align}
Since $ \psi_{|[0,t_0]}=\vph_{|[0,t_0]}$, 
for any point
$$x\in\mathcal{K}_2:=\{z:d(z,\ggm_\vph)=|z-\vph(s)| \text{ for some } s\in [0,t_0]\}$$
it holds that
$$d(x,\ggm_\psi) \leq d(x,\ggm_{\psi_{|[0,t_0]}}) = d(x,\ggm_{\vph_{|[0,t_0]}}),  $$
thus
\begin{equation}\label{1-1 est 4}
  \int_{\mathcal{K}_2} d(x,\ggm_{\psi} )^\p \d\mu
\leq \int_{\mathcal{K}_2} d(x,\ggm_{\vph} )^\p \d\mu .
\end{equation}
Given that for any point
$$x\in\mathcal{K}_3:=\{z:d(z,\ggm_\vph)=|z-\vph(s)| \text{ for some } s\in [t_0,t_1]\}$$
it holds that
$$d(x,\ggm_\psi)\leq |x-\psi(t_0)|= |x-\vph(t_0)| \overset{|\vph'|\equiv 1\ \L^1\text{-a.e.}}{\leq} 
d(x,\ggm_{\vph}) + \xi,$$
integrating with respect to $\mu$ implies, in the view of \eqref{p lem},
\begin{equation}\label{1-1 est 5}
  \int_{ \mathcal{K}_3} d(x,\ggm_{\psi } )^\p \d\mu
- \int_{\mathcal{K}_3 } d(x,\ggm_{\vph} )^\p \d\mu  \leq \xi {\p\big(D+ \xi\big)^{\p-1}} \mu(\mathcal{K}_3) .
\end{equation}
Since $\R^d=\mathcal{K}_1\cup \mathcal{K}_2
\cup \mathcal{K}_3$, combining \eqref{1-1 est 3}, \eqref{1-1 est 4} and
\eqref{1-1 est 5} gives
\begin{equation}\label{1-1 est 6}
  \int_{\R^d} d(x,\ggm_{\psi} )^\p \d\mu
- \int_{\R^d} d(x,\ggm_{\vph} )^\p \d\mu  \leq 2{\p\big(D+ \xi\big)^{\p-1}}\xi  .
\end{equation}
By construction $L(\psi)=L(\vph)+\xi$. Since
$\xi$ can be chosen arbitrarily small, it follows
 $$E_\mu^\la(\psi)-E_\mu^\la(\vph)\leq (2\vartheta+\la)\xi  ,$$
proving the claim.

\medskip

 Claim 2.
 \grn There exists $t_1>t_0$, with $t_1-t_0$, and hence $\xi$, arbitrarily small such that \nc
\[
\int_{0}^{L(\vph)} |\kappa_\vph|^2 \d\L^1-\int_{0}^{L(\vph)+\vep} |\kappa_\psi|^2 \d\L^1
\geq \frac12 M^2\xi.\]

 Note that $\psi_{|[0,t_0]}=\vph_{|[0,t_0]}$,
while $\psi_{|[t_1+\xi,L(\vph)+\xi]}$ 
differs from $\vph_{|[t_1,L(\vph)]}$ by a translation. Thus
\begin{equation}\label{1-1 est 1}
\int_0^{t_0} |\kappa_\psi|^2 \d\L^1 = \int_0^{t_0} |\kappa_\vph|^2 \d\L^1,
\quad \int_{t_1+\xi}^{L(\vph)+\xi} 
|\kappa_\psi|^2 \d\L^1 = \int_{t_1}^{L(\vph)} |\kappa_\vph|^2 \d\L^1.
\end{equation}
 Moreover 
\begin{align*}
\int_0^{L(\vph)+\xi} |\kappa_\psi|^2 \d\L^1
& = \int_0^{t_0} |\kappa_\psi|^2 \d\L^1
+\int_{t_0}^{t_1+\xi} |\kappa_\psi|^2 \d\L^1
+\int_{t_1+\xi}^{L(\vph)+\xi} |\kappa_\psi|^2 \d\L^1\\
&\overset{\eqref{1-1 est 1}}{=}
\int_{0}^{t_0} |\kappa_\vph|^2 \d\L^1
+\int_{t_0}^{t_1+\xi} |\kappa_\psi|^2 \d\L^1
+\int_{t_1}^{L(\vph)} |\kappa_\vph|^2 \d\L^1,
\end{align*}
therefore
$$\int_{0}^{L(\vph)} |\kappa_\vph|^2 \d\L^1-\int_{0}^{L(\vph)+\xi} |\kappa_\psi|^2 \d\L^1
=\int_{t_0}^{t_1} |\kappa_\vph|^2 \d\L^1 -
\int_{t_0}^{t_1+\xi} |\kappa_\psi|^2 \d\L^1.$$
By construction, it holds
$$ \int_{t_0}^{t_1+\xi} |\kappa_\psi|^2 \d\L^1 =
\frac12 \int_{t_0}^{t_1} |\kappa_\vph|^2 \d\L^1.$$
\grn Using H\"older inequality, by the definition of $M$ there exists $t_1>t_0$, with $t_1-t_0$ arbitrarily small such that  \nc
$$ \int_{t_0}^{t_1} |\kappa_\vph|^2 \d\L^1 \geq 
 \frac{|\vph'(t_1)-\vph'(t_0)|^2}{|t_1-t_0|} 
\geq 
M^2 \xi,$$
and thus
\begin{equation}\label{1-1 est 7}
\int_{0}^{L(\vph)} |\kappa_\vph|^2 \d\L^1-\int_{0}^{L(\vph)+\xi} |\kappa_\psi|^2 \d\L^1
\geq \frac12 M^2\xi,
\end{equation}
and the claim is proven.

\bigskip

Combining Claims 1. and 2. with the minimality of $\vph$ gives
$$0\geq \EE(\vph)-\EE(\psi ) \geq \frac12 \vep M^2\xi- (2\vartheta+\la )\xi ,$$
hence
\begin{equation*}
M\leq \sqrt{ \frac{2(2\vartheta+\la) }{\vep} },
\end{equation*}
and the proof is complete.
\end{proof}

Now we investigate the behavior of estimate \eqref{Y}
under scaling.
Let $\mu$, $\la$, $\vep$, $\p$ be given,
and let $\vph$ be a minimizer of $\EE$. Endow $\R^d$
with an orthogonal coordinate system, and consider the linear map
$$T:\R^d\lra \R^d,\qquad T(x):=rx,$$ 
where $r>0$ is a given homothety ratio. 
Set 
$$\vph_r: [0,rL(\vph)]\lra \R^d,\qquad \vph_r(t):=r\vph(t/r).$$
Note that
\begin{align*}
\int_{\R^d} d(x,\ggm_{\vph_r})^\p \d T_\sh\mu &= r^{ \p-d} \int_{\R^d} d(x,\ggm_{\vph})^\p \d\mu,\\
 L(\vph_r) &= rL(\vph), \\ 
\int_0^{L(\vph_r)}|\kappa_{\vph_r}|^2\d\L^1 & =\frac1r \int_0^{L(\vph)}|\kappa_{\vph}|^2\d\L^1,
\end{align*}
hence
\begin{align*}
r^{ \p-d} \EE(\vph) & =
r^{ \p-d}
\left(\int_{\R^d} d(x,\ggm_{\vph})^\p \d\mu + \la L(\vph) + \vep \int_0^{L(\vph)}|\kappa_{\vph}|^2\d\L^1
\right)\\
& =
\int_{\R^d} d(x,\ggm_{\vph_r})^\p\d T_\sh\mu
+r^{ \p-d-1}\la L(\vph_r) + r^{  \p-d+1}\vep  \int_0^{L(\vph_r)}|\kappa_{\vph_r}|^2\d\L^1\\
&=
\E_{T_\sh\mu}^{\la_r, \vep_r, \p}(\vph_r),
\end{align*}
where
$\la_r:=r^{ \p-d-1}\la$, $\vep_r:= r^{\p-d+1}\vep$.
Since $\vph$ is a minimizer for $\EE$, it follows that
$\vph_r$ is a minimizer for $\E_{T_\sh\mu}^{\la_r, \vep_r, \p}$.
Note that 
$\diam\supp(T_\sh\mu) =r\cdot\D $.
Theorem \ref{1-1} gives that $\vph_r'$ is $Y_r$-Lipschitz continuous,
where
$$Y_r:=\sqrt{ \frac{2\left(2\p \big( 2r\D\big)^{\p-1} +\la_r \right)}{\vep_r} }.$$
Note that 
\begin{equation}\label{scale}
Y_r=\frac1r \sqrt{ \frac{2\left(2\p \big(2\D\big)^{\p-1} + \la \right)}{\vep} } =\frac Yr,
\end{equation}
with $Y$ defined in \eqref{Y}. By construction, $\vph_r'(t)=\vph_r'(t/r)$ for any $t$,
 therefore the respective Lipschitz constants satisfy
 $Y_r=Y/r$, which is compatible with
\eqref{scale}. 

\subsection{Heuristic arguments}
We present some heuristic arguments about the sharpness
of the estimate \eqref{Y},
which suggest that such estimate has optimal order in $\vep$,
for small values of $\vep$.
Choose $\la$ such that $2\pi\la<(2\pi+1)$, and let $\p:=1$.
For $r>0$, let $S_r:=\{(x,y)\in \R^2:x^2+y^2=r^2\}$,
and 
$\mu:=\H^1_{\llcorner S_r}$.
We compare the following competitors:
\begin{listi}
\item let $\psi_r\in \C$ be a parameterization of $S_r$,
\item let $\psi_r^*\in C$ be a parameterization of the segment 
$\{(1-s)(-r,0)+s(r,0):s\in [0,1]\}$.
\end{listi}
Note that
$$\int_{\R^2} d(x,\ggm_{\psi_r})\d\mu = 0,\qquad
L(\psi_r)=2\pi r,\qquad
\int_0^{L(\psi_r)} |\kappa_{\psi_r}|^2 \d\L^1 = \frac{2\pi}{r}, $$
$$\int_{\R^2} d(x,\ggm_{\psi_r^*})\d\mu = r,\qquad
L(\psi_r^*)=2 r,\qquad
\int_0^{L(\psi_r^*)} |\kappa_{\psi_r^*}|^2 \d\L^1 = 0. $$
Thus it follows
$$\E_\mu^{\la,\vep,1}(\psi_r)= 2\pi\la r+ \frac{2\pi\vep}r,
\qquad
\E_\mu^{\la,\vep,1}(\psi_r^*)= r(1+2\la).$$
Note that for large $r$, $\psi_r$ is more convenient than
$\psi_r^*$, since $2\pi\la r$ dominates $2\pi\vep/r$, and
hypothesis $2\pi\la<(2\pi+1)$
implies $\E_\mu^{\la,\vep,1}(\psi_r)<\E_\mu^{\la,\vep,1}(\psi_r^*) $.
For $r\ra 0^+$, the term $2\pi\vep/r$ diverges. For
$r \simeq \sqrt{\vep}$ (i.e. $r$ and
$\sqrt{\vep}$ differ by a multiplicative constant), it holds
$$\E_\mu^{\la,\vep,1}(\psi_r)\simeq  2\pi(1+\la) \sqrt{\vep},
\qquad \E_\mu^{\la,\vep,1}(\psi_r^*) \simeq \sqrt{\vep}(1+2\la),$$
hence $\E_\mu^{\la,\vep,1}(\psi_r)$ and $\E_\mu^{\la,\vep,1}(\psi_r^*)$
have the same order in $\vep$.
Thus the maximum radius $r_0$ such that for $r<r_0$, $\psi_r^*$ is
more advantageous than $\psi_r$, has order
$r_0\simeq \sqrt{\vep}$. The derivative
$\psi_{r_0}'$ 
is approximately (upon constants independent of $\vep$)
 $\sqrt{1/\vep}$-Lipschitz continuous.

\bigskip

We present some heuristic arguments about the choice of parameters $\la,\vep$,
under given $p\geq 1$.
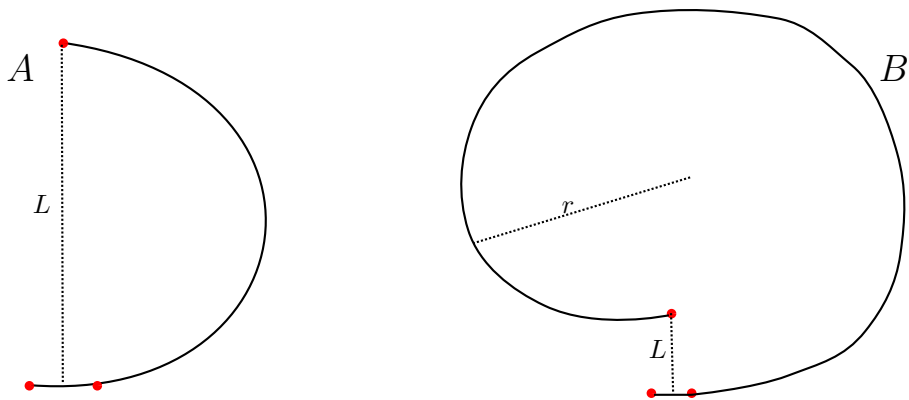
\begin{figure}[h!]
\includegraphics[scale=1.0]{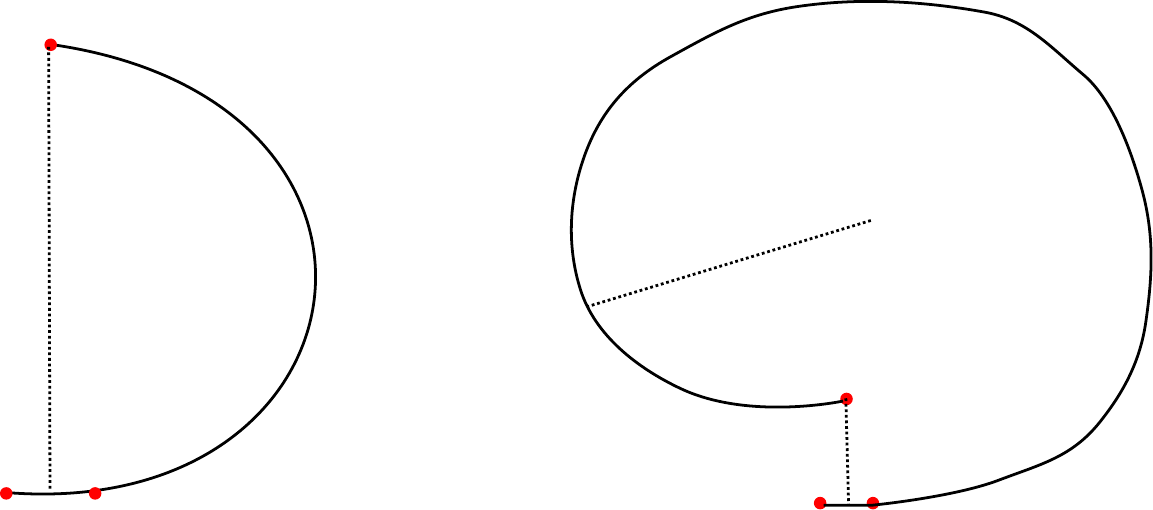}
\put(-330,70){$L$}
\put(-97,15){$L$}
\put(-130,70){$r$}
\put(-340,120){\LARGE{$A$}}
\put(-10,120){\LARGE{$B$}}
\caption{A schematic representation of two possible choices of minimizers.
Note the different scale in the two figures. The masses are concentrated on red points.}
\label{heu}
\end{figure}
Let $\mu$ be a measures concentrated on the red points in Figure
\ref{heu}, and let $\vph$ be the minimizer (denoted by black solid
lines in Figure \ref{heu}). In the following arguments
the symbol ``$\simeq$'' will denote ``equal upon universal constants''.
For the sake of simplicity we omit any non influential constant.

Configuration $A$:
 $\la L(\vph)  \simeq \la L$,
 $\vep\int_0^{L(\vph)}| \kappa_\vph|^2\d\L^1 \simeq
\vep /L$,
 $\int_{\R^d} d(x,\ggm_\vph)^p \d\mu $ is infinitesimal
compared to the above quantities,
since $\sm\supseteq\supp(\mu)$.
Total energy of configuration $A$: approximately
$\la L+\vep/L$.

\medskip

Configuration $B$: 
 $r\geq L$,
 $\la L(\vph)  \simeq \la r$,
 $\vep\int_0^{L(\vph)}| \kappa_\vph|^2\d\L^1 \simeq
\vep /r$,
 $\int_{\R^d} d(x,\ggm_\vph)^p \d\mu $ is infinitesimal
compared to the above quantities,
since $\sm\supseteq\supp(\mu)$.
Total energy of configuration $B$: approximately
$\la r+\vep/r$. Direct computing gives
the optimal value $r=\sqrt{\vep/\la}$, corresponding to an
approximate value of $\sqrt{\la\vep}$ for the energy.

Note that configuration $B$ is ``less desirable''
in data parameterization, since the minimizer
contains points further away from $\supp(\mu)$. Moreover,
a necessary condition for configuration $B$ to be preferable
to configuration $A$ is
$\sqrt{\vep/\la}>L$, i.e. $\vep>L^2 \la$.
Thus, for $\vep<L^2 \la$, configuration $A$ is preferable.

\subsection{Consequences of Theorem \ref{1-1}}
Recall that Problem \ref{main} has been introduced to overcome 
the fact that minimizers of Problem \ref{main1} can fail to be $C^1$
regular, which is undesirable
for data parameterization. Corollaries \ref{points 0}
and \ref{mass proj} prove that if $\mu\ll\L^d$, then minimizers of Problem \ref{main}
do not exhibit such undesirable properties. Some preliminary discussion
is required.
For given $\sm\in \A$ (defined in Problem \ref{main pen}),
define the set-valued ``projection'' map as
\begin{equation}\label{projection}
\Pi_{\sm}:\R^d\lra \mathcal{P}({\sm}) ,\qquad
\Pi(x)=\{y\in\sm:  |x-y|=d(x,\sm) \},
\end{equation}
with $\mathcal{P}(\sm)$ denoting the power set of $\sm$.
We recall that in \cite{ManMen} it has been proven that
the ridge
$$\mathcal{R}_\sm:=\{x: \exists y,z\in \sm,\ y\neq z,\ |x-y|=|x-z|=d(x,\sm)\}=
\{x: \sh\Pi_\sm(x) \geq  2 \} $$
is $\H^{d-1,1}$ rectifiable. 
In particular, if $\mu\ll \L^d$, for any $\sm\in \A$
the ridge $\mathcal{R}_\sm$ is $\mu$-negligible, thus $\sh\Pi_\sm(x)=1$
 $\mu$-a.e. In this case it is possible to define, for $\mu$-a.e. point, the point-valued
 function
 \begin{equation}\label{proj}
 \pi_\sm: \R^d\lra \sm,\qquad
 \pi_\sm(x):=y,\quad y \text{ is the unique point of } \sm \text{ satisfying }
 \{y\}=\Pi_\sm(x).
 \end{equation}   

\begin{defn}
Given $\mu\ll\L^d$, $\sm\in \A$ and
a subset $B\sse\sm$, the quantity $\mu\big(\pi_\sm^{-1}(B) \big)$ will be referred to
as mass projecting on $B$ in $\sm$. 
\end{defn}
For the sake of brevity we will omit writing {\em ``in $\sm$''}  if no risk of confusion
arises.
Note that, for any curve $\psi\in \C$, its { image} $\ggm_\psi$
belongs to $\A$. 
 
 \begin{cor}\label{points 0}
 Given $d\geq 2$, a measure $\mu\ll \L^d$, parameters $\la,\vep>0$, 
 $\p\geq 1$, a minimizer
$\vph\in \argmin_\C \EE$ and a time $t\in \big(0,L(\vph)\big)$, it holds
$\mu\Big(\pi_{\ggm_\vph}^{-1}\big(\{\vph(t)\}\big) \Big)=0$, i.e.
 the mass projecting on $\{\vph(t)\}$ is zero. 
 \end{cor}
 
\begin{proof}
Note that for sufficiently small $\eta$,
 the behavior of the curve $\vph_{|[t-\eta,t+\eta]}$ is locally approximated (in first order,
 upon translation and scaling) by
  the tangent derivative $\vph'(t)$. Thus the set 
  $$\{x\in \R^d:
 d(x,\ggm_\vph)=|x-\vph(t)|\} = \pi_{\ggm_\vph}^{-1}(\{\vph(t)\})$$ 
 is contained in $\Delta_t$, defined as the (unique) $(d-1)$-hyperplane
passing through $\vph(t)$ and orthogonal to $\vph'(t)$. Since $\mu\ll \L^d$,
it follows $\mu(\Delta_t)=0$, concluding the proof.
\end{proof}

\begin{cor}\label{mass proj}
Given $d\geq 2$, $q\in [1,+\8]$, a measure $\mu\ll \L^d$ with Radon-Nikodym derivative
$\d\mu/\d\L^d\in L^q$, parameters $\la,\vep>0$, $\p\geq 1$, a minimizer
$\vph\in \argmin_\C \EE$ and a time interval 
$[t_0,t_1]\sse \big(0,L(\vph)\big)$, then the mass projecting on $\vph{([t_0,t_1])}$
satisfies
\begin{equation}\label{est mass proj}
\mu\Big( \pi_{\ggm_{\vph}}^{-1}\big(\vph([t_0,t_1])\big) \Big)
\leq  \left\|\frac{\d\mu}{\d\L^d} \right\|_{L^q}
\left(
|t_1-t_0|\omega_{d-1} \Big( \frac{D^d}{d} +  \frac{D^{d+1}}{d+1} 
Y\Big)
\right)^{1-1/q},
 \end{equation}
 where
 $$Y=Y\big(\!\D,\la,\p,\vep\big):=
\sqrt{ \frac{2(2\vartheta+\la)}{\vep} } ,$$
$ \vartheta:=\p D^{\p-1}$, $ D:=2\D $,
 $\omega_{d-1}$
denotes the $\L^{d-1}$ measure of the $d-1$ dimensional unit ball, 
and $1/q:=0$ if $q=+\8$.
\end{cor}


\begin{proof} 
Let $\xi:=|t_1-t_0|$. For any $t$, the set $\pi_{\ggm_\vph}^{-1}\big(\{\vph(t)\}
\big)$ is contained
in $\Delta_t$, defined as the $(d-1)$-hyperplane passing through $\vph(t)$ 
and orthogonal to $\vph'(t)$. Thus it holds
$$ \pi_{\ggm_\vph}^{-1}\big(\vph([t_0,t_1])\big) \sse \bigcup_{t\in [t_0,t_1]} \Delta_t,$$
and it suffices to estimate
the $\mu$-measure of the right-hand side term.

Endow $\R^d$ with the standard orthonormal base $\{e_j\}_{j=1}^d$,
 $e_{j}:=(0,\cdots,0,\underbrace{1}_{j\text{{-th position}}},0,\cdots,0)$.
For $t\in [t_0,t_1]$, define the linear map $X_t:\R^d\lra \R^d$ as follows:
\begin{enumerate}
\item if $e_1=\vph'(t)$, then $X_t$ is the identity map,
\item if $e_1=-\vph'(t)$, then $X_t(p):=-p$ for any $p\in \R^d$,
\item if $\angle e_1\vph'(t)\in (0,\pi)$, let $X_t$
be the unique rotation by an angle $\angle e_1\vph'(t)\in (0,\pi)$, mapping $e_1$ in $\vph'(t)$. 
\end{enumerate}
Note that for any $t$, $X_t:\R^d\lra \R^d$ is a conformal isometry. Since Theorem \ref{1-1}
proves $C^{1,1}$ regularity for $\vph$, and
$Y$-Lipschitz continuity
for $\vph'$,
the mapping $t\mapsto X_te_j$ is $C^1$ regular for any $j$.
Moreover, it holds
\begin{align}
(\fal t^*) (\fal j)  \quad \Bigg|\left.\frac{\d}{\d t}X_te_j\right|_{t=t^*} \Bigg|
& \leq \lim_{h\ra 0} \frac{1}{h} |X_{t^*+h}e_j-X_{t^*}e_j| \notag\\
 & = \lim_{h\ra 0} \frac{1}{h} |X_{t^*+h}e_1-X_{t^*}e_1| \notag\\
 & = \lim_{h\ra 0} \frac{1}{h} |\vph'(t^*+h)-\vph'(t)|
  \leq Y . \label{est j}
\end{align}
Consider the moving frame $\{X_te_1,\cdots , X_te_d\}$. By construction,
$\Delta_t= \Span(X_t e_2,\cdots ,X_t e_d)+ \vph(t) $.

Let $P=\vph(t_0)+\sum_{j=2}^d P_j X_{t_0} e_j\in \Delta_{t_0}$ be an
arbitrarily given point, and define the trajectory 
$$\gm_P:[t_0,t_1]\lra \R^d,\qquad  \gm_P(t):= \vph(t)+ \sum_{j=2}^d P_j X_{t} e_j. $$
Since $t\mapsto X_te_j$ is $C^1$ regular for any $j$, $\gm_P$ is $C^1$ regular
for any $P\in \Delta_{t_0}$.
Note that, in general, $\gm_P$ is not parameterized by arc-length.
By construction, it holds
$$(\fal t) \quad
\Delta_t=\bigcup_{(P_2,\cdots,P_d)\in \R^{d-1} } \Big(\vph(t)+ \sum_{j=2}^d P_j X_{t} e_j \Big),$$
and the union of all trajectories covers $\bigcup_{t\in [t_0,t_1]} \Delta_t$, i.e.
$$\bigcup_{P\in \Delta_{t_0}}\gm_P([t_0,t_1])= \bigcup_{t\in [t_0,t_1]} \Delta_t.$$
Lemma \ref{comp} gives $\ggm_\vph\sse \big(\!\supp(\mu)\big)_{\D+(\D)^\p/\la}$,
thus
$$\sup_{x\in \supp(\mu),\ y\in \ggm_\vph} |x-y|\leq 2\D+\big(\!\D\big)^\p/\la=D. $$
Since $X_t$ is an isometry, it follows
\begin{equation}\label{traj full}
\supp(\mu)\cap \bigcup_{P\in \Delta_{t_0}}\gm_P([t_0,t_1])=\supp(\mu)\cap \bigcup_{t\in [t_0,t_1]} \Delta_t
\sse \bigcup_{\substack{P\in \Delta_{t_0},\\ |P-\vph(t_0)|\leq D}}\gm_P([t_0,t_1]).
\end{equation}
To estimate the length of the trajectory $t\mapsto\gm_P(t)$, direct computation gives
\begin{align}
 \int_{t_0}^{t_1} |\gm_P'(t)| \d t & =  \int_{t_0}^{t_1} \left|
 \vph'(t)+ \sum_{j=2}^d P_j \left.\frac{\d}{\d s}X_{s} e_j\right|_{s=t} \right| \d t 
  \leq \int_{t_0}^{t_1} |\vph'(t)|\d t + \int_{t_0}^{t_1}\left|\sum_{j=2}^d  P_j
  \left.\frac{\d}{\d s}X_{s} e_j\right|_{s=t} \right|\d t \notag\\
& = \xi +  \int_{t_0}^{t_1}\sqrt{ \sum_{j=2}^d  P_j^2 \left|\left.\frac{\d}{\d s}X_{s} 
\right|_{s=t} \right|^2 }
\d t 
 \overset{\eqref{est j}}{\leq}
\xi +  \int_{t_0}^{t_1}\sqrt{ Y^2 
\sum_{j=2}^d  P_j^2   }\ \d t \notag
\\
& = \xi +  \xi|P-\vph(t_0)|Y . \label{est traj}
\end{align}
This proves that the length of trajectory $t\mapsto \gm_P(t)$ depends only on
$|P-\vph(t_0)|$. Thus  
\begin{align*}
\L^d\Bigg( \supp(\mu)\cap  \bigcup_{t\in [t_0,t_1]} \Delta_t \Bigg)
& \overset{\eqref{traj full}}{\leq}
\L^d\Bigg( \bigcup_{{P\in \Delta_{t_0},\ |P-\vph(t_0)|\leq D}}\gm_P([t_0,t_1]) \Bigg) \\
& = \int_{\{P\in \Delta_{t_0}: |P-\vph(t_0)|\leq D\}} \int_{t_0}^{t_1} |\gm_P'(t)| \d t\d P \\
& \overset{  \eqref{est traj}}{\leq}
\xi\int_{\{v\in \Delta_{t_0}: |P-\vph(t_0)|\leq D\}} 
\left( 1+ |P-\vph(t_0)|Y\right) \d P \\
& = 
\xi \int_0^D \int_{\{P\in \Delta_{t_0}: |P-\vph(t_0)|=r\}} 
\left( 1+rY\right) \d P\d r\\
& = 
\xi\int_0^D  \omega_{d-1} r^{d-1} \left( 1+rY \right)  \d r\\
&=
\xi\omega_{d-1} \left( \frac{D^d}{d} +  \frac{D^{d+1}}{d+1} Y\right) .
\end{align*}
Since by hypothesis ${\d\mu}/{\d\L^d}\in L^q$, H\"older inequality gives
\begin{align*}
\mu\left( \supp(\mu)\cap  \bigcup_{t\in [t_0,t_1]} \Delta_t \right)
& \leq
\left\|\frac{\d\mu}{\d\L^d} \right\|_{L^q} 
\left(\L^d\Big( \supp(\mu)\cap  \bigcup_{t\in [t_0,t_1]} \Delta_t \Big)\right)^{1-1/q}\\
&\leq \left\|\frac{\d\mu}{\d\L^d} \right\|_{L^q}
\left(
\xi\omega_{d-1} \Big( \frac{D^d}{d} +  \frac{D^{d+1}}{d+1} Y\Big)
\right)^{1-1/q},
\end{align*}
concluding the proof.
\end{proof}

\section{Appendix: an ode for curvature.}\label{ode curve}
Here we derive the Euler-Lagrange equation for the energy $\E$, when the reference measure
$\mu$ has finitely many atoms.  
Let $\gm:[0,L]\lra \R^d$ be an arc-length parameterized minimizer of $\E^{\la,\vep,\p}_\mu$.
Let $\vph:[0,L]\lra \R^d$ be a $C^2$ regular perturbation. 
\grn We  show that, in $\R^2$,
the Euler-Lagrange equation 
is \grn up to parameters, the same as \nc  the Euler-Lagrange equation for elastica  functional, \cite{Levien}: \nc
\begin{equation*}
\kappa''+{\color{black}C_1}\kappa^3+{\color{black}C_2}\kappa=0.
\end{equation*}
 This is expected, as minimizers of $\E$ are 
``elastica-like''.
\medskip

We recall that elastica curves are solutions of the so-called ``elastica problem'',
one of the earliest examples of nonlinear displacement problems,
first proposed by Bernoulli in 1691.
%
%
%
%
%
%
There the integrated squared curvature  
 quantifies the elastic stress, while the average-distance term describes the ``pulling force'' of the weight hanging
on the elastica.

\medskip
 We first estimate the first variation of the curvature term
$ K(\gm)=\int_0^{L} |\gm''|^2 \d s$.
Given $h>0$, 
\begin{equation*}
K(\gm+h\vph)-K(\gm)
=\int_0^{L} |\gm'+h\vph'|^{-1}\bigg| \frac{\d}{\d s}\bigg(\frac{\gm'+h\vph'}{|\gm'+h\vph'|} \bigg) \bigg|^2\d s.
\end{equation*}
Direct computation gives
\begin{align}
\frac\d{\d h}K(\gm+h\vph)\Bigg|_{h=0}&=\int_0^L \Big[-\langle \gm',\vph'\rangle |\gm''|^2 +
2\Big\langle \gm'', \vph''-\gm''\langle \gm',\vph'\rangle-\langle\vph',\gm''\rangle\gm'-\langle\vph'',\gm'\rangle\gm'\Big\rangle
\Big]\d s\notag\\
&=\int_0^L [3(|\gm''|^2\gm')'+2\gm'''' ]\vph\d s + 2\gm''\vph''\bigg|_0^L,
\label{first variation of curvature}
\end{align}
{\color{black}Now we estimate the first variation of the length:
	\begin{align*}
	\la \int_0^L [| \gamma' +h  \vph' | - |\gamma' |]\d x &=
	\la \int_0^L\frac{h  \gamma'  \vph' +O(h ^2)}{| \gamma' +h  \vph' | + |\gamma' |} \d x.
	\end{align*}
	Without loss of generality, we take $\gamma$ to be arc-length parameterized, i.e.
	$|\gamma'|=1$ a.e.
	Dividing by $h $ and then taking the limit $h \to 0$ gives
	\begin{align*}
	\la \int_0^L\frac{\gamma'  \vph' }{2|\gamma' |} \d x
	&=\frac{\la}{2}\int_0^L\gamma'' \vph\d x
	+\frac{\la}{2} \gamma''\vph\bigg|_{0}^L.
	\end{align*}
Moreover,
	recall that all the masses projected on the knots, hence the 
average distance term remains unchanged.
	Combining with \eqref{first variation of curvature} and }
	 considering perturbations $\varphi$ which are compactly supported within $(0,L)$	gives 
{\color{black}
	\begin{align*}
\int_0^L \Big[\frac{\la}{2}\gamma'' + \vep[3(|\gm''|^2\gm')'+2\gm'''' ] \Big]\vph\d x=0,
	\end{align*}
	i.e.,
	\begin{align}
	\label{c contrib}
	\frac{\la}{2}\gamma'' + \vep[3(|\gm''|^2\gm')'+2\gm'''' ] =
	\frac{\la}{2}\gamma'' + \vep[3( 2\langle\gm'',\gamma'''\rangle
	\gamma'+  |\gm''|^2\gm'')+2\gm'''' ] 
	=0.
	\end{align}
In $\R^2$, using Frenet-Serret formulas, we have
\begin{align}
\gm'=\overset\to{t},&\qquad \gm''=\kappa \overset\to{n}, 
\label{fs1}\\
\gm'''=\kappa'\overset\to{n}-\kappa^2\overset\to{t},&\qquad
\gm''''=\kappa''\overset\to{n}-3\kappa\kappa'\overset\to{t}-\kappa^3\overset\to{n},\label{fs2}
\end{align}
hence \eqref{c contrib} becomes
\begin{align*}
	0&=\frac{\la}{2} \kappa \overset\to{n} +\vep \Big[ 6 \langle
	\kappa \overset\to{n}, \kappa'\overset\to{n}-\kappa^2\overset\to{t} \rangle\overset\to{t}
	+2 (\kappa''\overset\to{n}-3\kappa\kappa'\overset\to{t}-\kappa^3\overset\to{n}) \Big]
		=\Big(\frac{\la}{2} \kappa  +2\vep [ \kappa''-\kappa^3 ]\Big)\overset\to{n},
\end{align*}
i.e.
\[\frac{\la}{2} \kappa  +2\vep [ \kappa''-\kappa^3 ]=0.\]
}

\medskip

 	{\color{black} Finally, note} the segment
 	$\gm((t_0,t_1))$ minimizes the (sum of) length and curvature terms, i.e.
 	\begin{align*}
 	 	 \gm_{|(t_0,t_1)} \in \argmin &\bigg\{ \la L(\sigma)+\vep \int_0^{ L(\sigma)} |\kappa_\sigma|^2\d s:\\
 	 	 & \;\;\,
 	\sigma \text{ is an arc-length parameterized curve with endpoints } \gm(t_0)\text{ and } \gm(t_1) \bigg\}.
 	 	\end{align*}

\subsection*{Acknowledgements} 
The authors are grateful to the referees for valuable suggestions. 
XYL acknowledges the support of NSERC Discovery Grant 
``Regularity of minimizers and pattern formation in geometric minimization problems''.
This research was partly carried out when XYL was affiliated with Instituto Superior T\'ecnico.
DS is grateful to  the National
Science Foundation for support under grants  DMS 1211161, CIF 1421502 and DMS 1516677.
Both authors are thankful to  the Center for Nonlinear Analysis (CNA)
 for its support.

%


\bibliographystyle{siam}
\bibliography{bibadp}

\end{document}